\documentclass[11pt]{amsart}
\usepackage[top=2.5cm, bottom=4.5cm, left=2.8cm, right=2.8cm]{geometry}

\usepackage[english]{babel}

\usepackage{amsmath}
\usepackage{amsthm}
\usepackage{amsfonts}
\usepackage{amssymb}
\usepackage{graphicx}
\usepackage{tikz-cd}
\usepackage[colorlinks=true, allcolors=black]{hyperref}
\usepackage{comment}
\usepackage{mathtools}
\usepackage[normalem]{ulem}

\newtheorem{theorem}{Theorem}[section]
\newtheorem{lemma}[theorem]{Lemma}
\newtheorem{prop}[theorem]{Proposition}
\newtheorem{corollary}[theorem]{Corollary}
\newtheorem{conj}[theorem]{Conjecture}
\theoremstyle{definition}
\newtheorem{setting}[theorem]{Setting}
\newtheorem{definition}[theorem]{Definition}
\newtheorem{remark}[theorem]{Remark}
\newtheorem{example}[theorem]{Example}

\newcommand{\bZ}{{\mathbb Z}}
\newcommand{\bC}{{\mathbb C}}

\newcommand{\ra}{\rightarrow}

\title{Bergman spaces on algebraic curves}

\thanks{The first two authors are partially supported by  ``\'Elvonal (Frontier)'' Grant KKP 144148,
RSz was supported by ERC Advanced Grant KnotSurf4d}

\author{László Koltai}
\address{HUN-REN A. R\'enyi Institute of Math.,
Re\'altanoda utca 13-15, H-1053, Budapest, Hungary and\newline
 \hspace*{3mm} ELTE - Eötvös L. Univ., Dept. of Geo.,
 P\'azm\'any P\'eter s\'et\'any 1/C, 1117, Budapest, Hungary}
\email{koltai.laszlo@renyi.hu}

\author{Alexander A. Kubasch}
\address{HUN-REN A. R\'enyi Institute of Math.,
Re\'altanoda utca 13-15, H-1053, Budapest, Hungary and\newline
 \hspace*{3mm} ELTE - Eötvös L. Univ., Dept. of Geo.,
 P\'azm\'any P\'eter s\'et\'any 1/C, 1117, Budapest, Hungary}
\email{kubasch.alexander@renyi.hu }

\author{Róbert Szőke}
\address{HUN-REN A. R\'enyi Institute of Math.,
Re\'altanoda utca 13-15, H-1053, Budapest, Hungary and\newline
 \hspace*{3mm} ELTE - Eötvös L. Univ., Dept. of Anal.,
 P\'azm\'any P\'eter s\'et\'any 1/C, 1117, Budapest, Hungary}
\email{rszoke@ttk.elte.hu}


\begin{document}
\begin{abstract}
    A theorem of Wiegerinck asserts that the Bergman space of an open subset of the complex numbers is either infinite-dimensional or trivial. Recently, this has been generalized to holomorphic vector bundles over the projective line by the third author and later to vector bundles over any compact Riemann surface by Gallagher, Gupta and Vivas.
    
    In the present paper we extend the above results to the case of certain singular metrics associated to divisors on a Riemann surface. As corollaries we obtain versions of Wiegerinck's theorem for both projective and affine algebraic curves.
\end{abstract}
\maketitle

\section{Introduction}

\noindent Bergman spaces feature prominently in complex geometry, functional analysis, operator theory, and even arise in mathematical physics as quantum Hilbert spaces. They have been studied extensively over the past half century, but (as far as the authors can tell) have never been considered over \emph{singular} spaces. This paper aims to be a first step in this direction by studying Bergman spaces over both projective and affine algebraic curves.\vspace{3mm} 

\noindent Let $U\subseteq \bC^n$ be an open subset. The   Bergman space $A^2(U)$ is the Hilbert space of holomorphic functions on $U$ which are square-integrable with respect to the Euclidean volume form $dV$. Biholomorphic open sets have isomorphic Bergman spaces, yielding an important geometric invariant. It is easy to see that $A^2(\bC^n)$ is trivial.  On the other hand, if $U$ is bounded, then $A^2(U)$ contains all complex polynomials, making it infinite-dimensional. Moreover, Wiegerinck showed in \cite{Wiegerinck} that for any $k \in \mathbb{N}$, there exists a (non Stein) domain $U_k \subseteq \bC^2$ having $k$-dimensional Bergman space.

In dimension $n=1$ however, the situation is strikingly different. In the same paper Wiegerinck proved the following dichotomy: if $U \subseteq \bC$ is any open subset, then $A^2(U)$ is either infinite-dimensional or trivial. This phenomenon is illuminated by earlier work of Carleson \cite{Ca83}: it turns out that the Bergman space of $U \subseteq \bC$ is trivial if and only if its complement is ``small", because in this case the restriction map $A^2(\bC) \to A^2(U)$ is onto and $A^2(\bC)$ is trivial. The precise meaning of ``smallness" is potential-theoretic in nature: a set is ``small" if and only if it is locally polar (see Definition \ref{def:locpol}).

In $\bC$ every open set is Stein, unlike in higher dimensions. It is then very natural to guess that the same dichotomy holds for all Stein domains $U \subseteq \bC^n$ and any $n$. This is the celebrated Wiegerinck conjecture. Despite significant work and some progress, it remains open. In this paper we pursue a different kind of generalization to Wiegerinck’s theorem.\vspace{3mm}

\noindent The notion of Bergman spaces can be extended easily to singular spaces. Suppose $X$ is a complex curve, i.e. a reduced, one-dimensional complex space. We say that $\sigma$ is a volume form on $X$ if it is a positive and continuous volume form on the smooth part $X^*$ and in a small open neighborhood of each singular point it is the restriction of the K\"ahler form of $\bC^n$ using some local embedding.  
 The Bergman space $A^2(X,\sigma)$ consists of those holomorphic functions on $X$ that are also in $L^2(X^*,\sigma).$

\begin{theorem}\label{intr:simaffin}
    \normalfont{\textbf{[Corollary \ref{cor:simaffin}]}} \textit{Let $M \subset \bC^n$ be a smooth affine algebraic curve with volume form $\sigma$ given by the restriction of the Kähler form on $\bC^n$. Given any open subset $U \subseteq M$, the Bergman space $A^2(U,\sigma)$ is either infinite-dimensional or trivial.}
\end{theorem}

\noindent In fact, even more is true: if $Y\subset \bC^n$ is an affine algebraic curve with only ``mild" singularities, then the statement of Theorem \ref{intr:simaffin} holds for $Y$ (see Corollary \ref{cor:szingWD}). This is not true in general however: in subsection \ref{ss:aff} we give an example of a plane algebraic curve $Y \subset \bC^2$ and an open subset $U\subseteq Y$ whose Bergman space is neither infinite-dimensional, nor trivial. It follows that the Wiegerinck conjecture (as stated above) cannot be extended to arbitrary Stein varieties. There is however a more general version of the Wiegerinck dichotomy which does hold for all affine algebraic curves, see Theorem \ref{intr:affmain}.

Our proof of Theorem \ref{intr:simaffin} relies heavily on the fact that the curves in question are \emph{algebraic}. Hence, we propose the following analytic version.

\begin{conj}
    Let $M$ be a non-compact Riemann surface together with an analytic embedding $M \hookrightarrow \bC^n$. Then the statement of Theorem~\ref{intr:simaffin} holds for $M$.
\end{conj}

\noindent While the formulation of Theorem 1.1 is a direct generalization of Wiegerinck’s theorem, its
proof relies on a projective version of the dichotomy.

Let $X \subset \mathbb{P}^n$ be a projective algebraic curve with volume form $\sigma$ and let $E \to X$ be a holomorphic vector bundle with Hermitian metric $h$. Given an open subset $U \subseteq X$, the Bergman space of square-integrable holomorphic sections is denoted by $A^2(U,\sigma,E,h)$. Our main objective is to study these Bergman spaces. Using the normalization of $X$, this problem essentially reduces to studying Bergman spaces over open subsets of a compact Riemann surface equipped with a certain ``singular" volume form. This leads us to considering the following situation.

Let $M$ be a compact Riemann surface together with a divisor $D \in \text{Div}(M)$. By a
$D$-volume form $\rho$ on $M$ we mean a volume form on $M \setminus \text{supp}(D)$ with ``vanishing orders" and ``pole orders" governed by $D$ (see Definition \ref{def:gyokosmetrika}). In this context, the Wiegerinck dichotomy takes on the following form:

\begin{theorem}\label{intr:main}\normalfont{\textbf{[Theorem \ref{thm:gyokos}]}}
    \textit{Let $M, D$ and $\rho$ be as above and let $E \to M$ be a holomorphic vector bundle with Hermitian metric $h$. Given an open subset $U \subseteq M$, then} 
        $$\dim A^2(U,\rho,E,h) < \infty \quad \Leftrightarrow \quad \begin{array}{c}
        M \setminus U\\
        \text{is locally polar}
        \end{array} \quad \Leftrightarrow \quad A^2(U,\rho,E,h) = H^0(M,E\otimes L)$$
        \textit{for some holomorphic line bundle $L \in \text{\normalfont Pic}(M)$ depending on both $D$ and $U$, but not on $E$.}
\end{theorem}

\noindent The line bundle $L$ accounts for the fact that the vanishing locus of $\rho$ on the boundary of $U$ allows for meromorphic $L^2$-sections, whereas the poles of $\rho$ force $L^2$-sections to have certain vanishing orders. This implies the obvious inclusion
$$H^0(M,E\otimes L) \subseteq A^2(U,\rho,E,h)$$
for all open subsets $U \subseteq M$. If $D=0$, the line bundle $L$ turns out to be trivial. In this special case, Theorem \ref{intr:main} was proven in \cite{Sz22, GGV22} for $M=\mathbb{P}^1$ and in \cite{GGV24} in the general case.

Now, let $X$ be a projective algebraic curve with volume form $\sigma$ and denote by $\pi : \widetilde{X} \to X$ its normalization. The pullback $\pi^*\sigma$ turns out to be a $D_\mathbf{m}$-volume form associated to a certain effective divisor $D_\mathbf{m} \in \text{Div}(\widetilde{X})$ depending on the singularities of $X$ (see Definition \ref{def:multdiv}). Applying Theorem \ref{intr:main} to $\widetilde{X}$ and $\pi^*\sigma$ yields

\begin{theorem}\label{intr:projmain}\normalfont{\textbf{[Theorem \ref{thm:projmain}]}}
    \textit{Let $X$ be a projective algebraic curve and let $E \to X$ be a holomorphic vector bundle with Hermitian metric $h$. Given an open subset $U \subseteq X$, then}
        $$\dim A^2(U,\sigma,E,h) < \infty \quad \Leftrightarrow \quad \begin{array}{c}
        X \setminus U\\
        \text{is locally polar}
        \end{array} \quad \Leftrightarrow \quad A^2(U,\sigma,E,h) = H^0(X,E\otimes \mathcal{P}_X)$$
        \textit{for some coherent analytic sheaf $\mathcal{P}_X$ depending on the singularities of $X$ in the closure of $U$.}
\end{theorem}

\noindent The sheaf $\mathcal{P}_X$ arises as the pushforward of a line bundle on a partial normalization of $X$, which is closely related to (but not the same as) the line bundle mentioned in Theorem \ref{intr:main}.\vspace{3mm}

\noindent The affine version of Theorem \ref{intr:projmain} is another application of Theorem \ref{intr:main}. Let $Y$ be an affine algebraic curve embedded into some Euclidean space $\bC^n$ and let $\sigma$ be the volume form on $Y$ given by the restriction of the Euclidean K\"ahler form. Consider the projective closure $X:=\overline{Y}$ and its normalization $\pi : \widetilde{X} \to X$. In this case the pullback $\pi^*\sigma$ turns out to be a $D_\mathbf{m}^\mathbb{A}$-volume form, where $D_\mathbf{m}^\mathbb{A}$ is a divisor on $\widetilde{X}$ depending on the singularities of $Y$ and the intersection of its closure with the hyperplane at infinity.

\begin{theorem}\label{intr:affmain}\normalfont{\textbf{[Theorem \ref{thm:affmain}]}}
    \textit{Let $Y$ be an affine algebraic curve with volume form $\sigma$. Given an open subset $U \subseteq X$, then}
        $$\dim A^2(U,\sigma) < \infty \quad \Leftrightarrow \quad \begin{array}{c}
        Y \setminus U\\
        \text{is locally polar}
        \end{array} \quad \Leftrightarrow \quad A^2(U,\sigma) = H^0(X,\mathcal{A}_X)$$
        \textit{for some coherent analytic sheaf $\mathcal{A}_X$ depending on the singularities of $Y$ in the closure of $U$ and at infinity.}
\end{theorem}

\noindent The key difference between the projective and affine cases is that $D_\mathbf{m}$ is effective, while $D_\mathbf{m}^\mathbb{A}$ always has some strictly negative coefficients. It is precisely this negativity which yields the triviality of the Bergman spaces in Theorem \ref{intr:simaffin}.\vspace{3mm}

\noindent \textbf{Structure of the paper.} In section 2 we show that the Bergman spaces associated to $D$-volume forms are indeed Hilbert spaces and prove Theorem \ref{intr:main}. Section 3 provides a brief introduction to the theory of curve singularities and discusses locally polar sets on singular curves. Finally, in section 4 we derive Theorems \ref{intr:projmain} and \ref{intr:affmain} and give a couple of examples. We also define an ``$L^2$-version" of the delta invariant and compare it to the usual delta invariant.\vspace{3mm}

\noindent \textbf{Acknowledgement.} The authors would like to thank the referee for their thorough review of the paper and very useful suggestions.

\section{Bergman spaces on Riemann surfaces}

\noindent In this section we study Bergman spaces over a Riemann surface $M$ that
 are associated to a 
 Hermitian holomorphic vector bundle $E\rightarrow M$ and a \emph{singular} volume form $\rho$ on 
 $M$. The nature of the singularities of $\rho$ will be specified later in this section.

\subsection{The Bergman space of a domain}
 \noindent We begin with the case when $M=U\subset\bC$ is an open subset and $E$ is the trivial line bundle.
Given a measurable function $\mu:U\rightarrow[0,\infty]$, let $L^2(U,\mu)$ be the Hilbert space of complex valued $\mu$-square integrable measurable functions on $U$, with the scalar product
$$
\langle f, g\rangle:=\int_U f(z)\overline{g(z)}\mu(z)\frac{idz\wedge d\bar{z}}{2}.
$$ 
  \begin{definition} 
  Denote by 
  $$A^2(U,\mu):=L^2(U,\mu)\cap\mathcal O(U),
  $$ the set of all $\mu$-square integrable holomorphic functions on $U$.  This space 
  is called the \emph{Bergman space} with weight $\mu$ (or the $\mu$-Bergman space) over $U$.
  \end{definition}

  \begin{remark}\label{rem:eltunos}
      When $\mu\equiv1$, it is simply denoted by $A^2(U)$. It is a closed subspace in $L^2(U)$ and thus a Hilbert space. Given a point $p \in U$ and an integer $m>0$, let
  $$
    A^2_m(U, p):=\{ \, f\in A^2(U) \ | \ f^{(j)}(p)=0, \, j=0,\dots, m \, \}.
$$
This is also a Hilbert space, as the evaluation maps $\text{ev}^j_p:f\mapsto f^{(j)}(p)$ are continuous linear functionals on 
$A^2(U)$ for all $p \in U$, cf. \cite{HKZ}.
  \end{remark}
For a general weight function $\mu$,
  $A^2(U,\mu)$ is still a linear subspace of  $L^2(U,\mu)$ (possibly equal to $\{0\})$,
but it is not necessarily a closed subspace i.e. not a Hilbert space. It is proved in \cite{HQW21} that $A^2(U,\mu)$ is closed in $L^2(U,\mu)$ iff the point evaluations are locally uniformly bounded. Pasternak-Winiarski in \cite{PW92} constructs a weight function  $\mu$ on the unit disk $D$, for which the point evalutaion at the origin is not continuous on  $A^2(D,\mu)$. Hence this is an example where the Bergman space is not a Hilbert space. On the other hand, under mild conditions on $\mu$, $A^2(U,\mu)$ is a Hilbert space as well, for instance when $1/\mu$ is a locally integrable function (\cite{PW90}, Proposition 2.1 and Theorem 3.1).

Fix an integer $m \in \bZ$. We will be interested in Bergman spaces of weight $\mu_m(z):=|z|^{2m}$ over an open subset $U \subseteq \bC$. When $m<0$, $1/\mu_m=|z|^{-2m}$
is continuous and hence locally integrable. When $m>0$, this does not hold, but 
$1/\mu$ is in $L^1_{\text{loc}}(U\setminus \{0\})$ and this suffices, see
(\cite{PW90}, Theorem 3.2). Therefore $A^2(U,\mu_m)$ is a Hilbert space for all integers $m$.\vspace{3mm}

\noindent For a subset $U\subset\bC$, let $U^*:=U\setminus \{0\}$. 

\begin{lemma}\label{lem:korlap}
     Let $0 \in U\subset\bC$ be open and  $m\in \mathbb{Z}$.
     \begin{enumerate}
        \item Suppose $m < 0$. Then the map
         $$\Phi_- : A^2(U,\mu_m) \to A^2(U), \quad f(z) \mapsto z^mf(z)$$
         is a surjective linear isometry. In this case we also have that
            $$ A^2(U,\mu_m)=A^2(U^*,\mu_m).$$
         \item Suppose $m > 0$. Then the map
         $$\Phi_+ : A^2(U,\mu_m) \to A^2_m(U,0), \quad f(z) \mapsto z^mf(z)$$
         is a surjective linear isometry.
     \end{enumerate}
     
\end{lemma}

\begin{proof} If $m=0$, the first map is the identity map and the second is the classical statement that isolated $L^2$ singularities are removable. The last equality also follows from this.
The cases $m\not=0$ follow easily from the definition and the $m=0$ cases. 
\end{proof}
\begin{corollary}\label{cor:L2sing}
 For
 every $m\in\mathbb Z$, the Hilbert space
 $A^2(U^*,\mu_m)=z^{-m}A^2(U)$.
\end{corollary}

\subsection{Bergman spaces on Riemann surfaces}
In the following we generalize Remark \ref{rem:eltunos} about Bergman spaces with prescribed vanishings to the case of Riemann surfaces.
 \begin{definition}\label{def:simabergman}
    Let $M$ be a Riemann surface with a holomorphic vector bundle $E\to M$ equipped with a continuous Hermitian metric $h$. Let $\Sigma\subset M$ be a discrete set and $\rho$ a positive, continuous volume form on $M\setminus\Sigma$. The associated Bergman space is
    $$A^2(M,\rho,E,h) := \Big\{\, s \in H^0(M,E) \ \Big| \ \|s\|_{L^2}^2:= \int_{M\setminus \Sigma} h(s,s)\rho < \infty \, \Big\}$$
    where $H^0(M,E)$ denotes the vector space of holomorphic sections of $E$.
\end{definition}

\noindent The behavior of $\rho$ in the points of $\Sigma$ is specified in subsection \ref{ss:divbergman}. Until then $\Sigma$ will be assumed to be empty. In this case the Bergman space is known to be a Hilbert space (see for example \cite{PW94}, although there both the volume form and the metric is assumed to be smooth, the proof works in the continuous category as well).

\begin{definition}\label{definition:div} 
Let $M$ be a Riemann surface, $E\rightarrow M$ a holomorphic vector bundle of rank $r$ and $S$  a meromorphic section of $E$. Let $s_1, \dots, s_r$ be the coordinate functions of $S$ in a local trivialization of $E$ around $p \in M$ and set $\text{ord}_pS:=\min\{\,\text{ord}_ps_k, \,k=1,\dots r\,\}$. This is independent of the chosen trivialization. Define the divisor of $S$ as
$\text{div}(S):=\sum_{p\in M}(\text{ord}_pS)\cdot p$.   
\end{definition}

 \begin{remark}\label{rem:divtenzor}
 If we have two holomorphic vector bundles $E_i\to M$ $(i=1,2)$ with meromorphic sections $S_i$, then
 $
 \text{div}(S_1\otimes S_2)=\text{div}(S_1)+\text{div}(S_2).
 $
 \end{remark}
 \begin{definition}
 Let $M$ be a Riemann surface, $E\rightarrow M$ a holomorphic vector bundle  with continuous Hermitian metric $h$, $D$ an effective divisor  and $\omega$ a positive continuous volume form on $M$.
 Set 
 $$
A^2_D(M,\omega,E,h):=\{\,S\in A^2(M,\omega,E,h) \ \mid \ \text{\normalfont div}(S)\ge D\, \}
$$
 \end{definition}
\begin{theorem}\label{theorem:embedding}  The space $A^2_D(M,\omega,E,h)$
is a closed subspace in  $A^2(M,\omega,E,h)$, hence a Hilbert space.   
\end{theorem}

\begin{proof}
For $D=0$, $A^2_D(M,\omega,E,h)=A^2(M,\omega,E,h)$, so we can assume that $D=\sum m_pp,$ with $m_p>0$. It is enough to prove the theorem when $D=m\cdot p$ for some $p\in M$, since
  $$
 A^2_D(M,\omega,E,h)=\bigcap_{p\, \in \, \text{supp}(D)}A^2_{mp}(M,\omega,E,h). 
  $$
  So let $D=mp$, with $m$ a positive integer. Let $r$ be the rank of $E$. Fix the following data: $V$  a small enough neighborhood of $p$ with
  a holomorphic frame $e_1,\dots,e_r$ of $E$ over $V$ and a local holomorphic coordinate system
  $\chi:(V,p)\rightarrow (\mathbb{B}(0,2\varepsilon),0)$.
  
  Let $S\in A^2(M,\omega,E,h)$ and write it as $\left.S\right|_V=\sum^r_1f_ke_k$,
  where $f_k\in\mathcal O(V)$.
  \begin{lemma} \label{lem:kiertfolyt} For
  every $k=1,\dots,r$ and $ j=0, 1, 2, \dots$, the  evaluation map
$$\text{\normalfont Ev}^j_{p,k}:A^2(M,\omega,E,h)\rightarrow \bC,
\qquad S\mapsto (f_k\circ\chi^{-1})^{(j)}(0)$$ is a continuous linear functional.
  \end{lemma}

\noindent Armed with this lemma the proof of Theorem \ref{theorem:embedding} easily follows since
$$
A^2_{mp}(M,\omega,L,h)=
\bigcap^r_{k=1}\bigcap^m_{j=0}\ker \text{Ev}^j_{p,k}.
$$
\end{proof}

  \begin{proof}[Proof of Lemma \ref{lem:kiertfolyt}]
      First we show that $f_k\circ\chi^{-1}\in A^2(\mathbb B(0,\epsilon))$.
      Let $g$ be the Hermitian metric on
      $\left. E\right|_V$ for which the frame $e_1,\dots,e_r$ is unitary, i.e. $g(e_j,e_l)=\delta_{jl}$. 
    Since $W:=\chi^{-1}(\mathbb B(0,\varepsilon))$ is a compact subset in $V$, there exists some constants 
     $0<c$, so that in $W$ we have $cg\le h$.  Then           
     \begin{equation}\label{E:lokbecsl}
     \begin{split}
\infty>\|S\|^2_{L^2}:=
\int_Mh(S,S)\omega
\ge\int_Wh(S,S)\omega\ge
c\int_Wg(S,S)\omega=
c\int_W\sum^r_{l=1}|f_l|^2\omega
 \\
\ge c\int_W|f_k|^2\omega
=c\int_{\mathbb B(0,\epsilon)}|f_k\circ\chi^{-1}|^2
(\chi^{-1})^*\omega=c
\int_{\mathbb B(0,\epsilon)}
|f_k\circ\chi^{-1}|^2 Fd\lambda,
\end{split}
     \end{equation}
     where $d\lambda$ is the area form on $\bC$ and the function $F$ is defined by the equation $(\chi^{-1})^*\omega=Fd\lambda$ on $\mathbb B(0,2\varepsilon)$. By its definition, $F$ is a smooth, positive  function. Thus there exist a  constant $0<c_1$ with $c_1\le \left.F\right|_{\mathbb B(0,\epsilon)}$. Hence from \eqref{E:lokbecsl} we get
 \begin{equation}\label{E:ball}     
  \int_{\mathbb B(0,\epsilon)}
|f_k\circ\chi^{-1}|^2d\lambda\le
\frac1{c_1}
\int_{\mathbb B(0,\epsilon)}
|f_k\circ\chi^{-1}|^2 Fd\lambda
\le \frac1{cc_1}\|S\|^2_{L^2}.
\end{equation}
Thus indeed $f_k\circ\chi^{-1}\in A^2(\mathbb B(0,\epsilon))$.
The point evaluation 
$$
\text{ev}^j_0:A^2(\mathbb B(0,\varepsilon))\rightarrow\bC
\qquad
\text{ev}^j_0(g):=g^{(j)}(0)
$$ is a continuous linear functional   (see for instance \cite{HKZ}).
Hence there exists a constant $0<K_j$, so that together with equation~\eqref{E:ball}
we obtain
\begin{align*}
    |\text{Ev}^j_{p,k}(S)|=& |(f_k\circ\chi^{-1})^{(j)}(0)|=\\
    =& |\text{ev}^j_0(f_k\circ\chi^{-1})|\le K_j\sqrt{\int_{\mathbb B(0,\epsilon)}
    |f_k\circ\chi^{-1}|^2 d\lambda}\le\frac{K_j}{\sqrt{cc_1}}\|S\|_{L^2}.
\end{align*}
\end{proof}

\subsection{Divisorial Bergman spaces}\label{ss:divbergman} In this subsection we introduce certain ``singular" volume forms associated to divisors on a Riemann surface. We show that the corresponding Bergman spaces are Hilbert spaces.

\begin{definition}\label{def:gyokosmetrika}
    Let $M$ be a Riemann surface and    $D = \sum_pm_p\cdot p \in \text{Div}(M)$ a divisor on $M$. A $D$-volume form $\rho_D$  on $M$ is defined as a positive continuous volume form on $M\setminus\text{supp}(D)$ with the following additional property. Near any point $p \in \text{supp}(D)$, $$\rho_D \sim c_p|z|^{2m_p}\frac{idz\wedge d\bar z}{2}$$ for some local coordinate system $z$ around $p$. Here $c_p$ is a positive constant and the symbol $\sim$ means that the two forms are asymptotically equal as $z \to 0$. Note that when $D=0$, a $D$-volume form is just a positive continuous volume form.
\end{definition}
\begin{remark}\label{rem:D-volume}
Let $\rho_D$ be a fixed $D$-volume form on the Riemann surface $M$ and $\eta:M\rightarrow\mathbb R$ a continuous, positive  function. Then $\eta\rho_D$ is also  a $D$-volume form and all $D$-volume forms can be obtained  this way. $D$-volume forms can be constructed as follows. 
 Let  $L_D\rightarrow M$ be the  holomorphic line bundle associated to $D$ and let $s_D$ be the corresponding meromorphic section with $\text{div} (s_D)=D$.
If $\kappa$ is a continuous Hermitian metric on $L_D$ 
 and $\omega$ is a positive continuous volume form on $M$ then $\rho_D:=\kappa(s_D,s_D)\omega$ is a $D$-volume form.
 
\end{remark}    

\begin{remark}\label{rem:rhoindep}
Suppose $M$ is a compact Riemann surface with a holomorphic vector bundle 
$E\rightarrow M$ equipped with a Hermitian metric $h$ and let $D \in \text{div}(M)$. Let $U \subseteq M$ be open. The Bergman space $A^2(U,\rho_D,E,h)$ is defined according to Definition \ref{def:simabergman} with $\Sigma = \text{supp}(D)$. It is independent of the metric $h$ and the choice of the $D$-volume form $\rho_D$ in light of Remark~\ref{rem:D-volume}. The  $L^2$ norms on this vector space arising from different Hermitian metrics and $D$-volume forms will all be equivalent. 
\end{remark}

\begin{setting}\label{S:input}
 Let $M$ be a Riemann surface, $\omega$ a positive continuous volume form on $M$,
 $(E,h)\rightarrow M$ an Hermitian holomorphic vector bundle with $h$ continuous and
 $D$ a divisor on $M$. Let
$D_+:=\sum_{m_p>0}m_pp$ and  
  $D_-:=\sum_{m_p<0}m_pp$.
    Then $D=D_++D_-$. Denote by  $L_{D_\pm}$ and  $L_D$ the associated holomorphic line bundles.  $L_D$
    is isomorphic to $L_{D_+}\otimes L_{D_-}$. Let $s_+:M\rightarrow L_{D_+}$ (resp. $s_-:M\rightarrow L_{D_-}$) be a holomorphic (resp. meromorphic) section with $\text{div}(s_\pm)=D_\pm$.
    Let $\kappa_+$ (resp. $\kappa_-$) be a continuous Hermitian  metric on $L_{D_+}$ (resp. on $L_{D_-}$).
    Then $\kappa=\kappa_+\otimes \kappa_-$ is a continuous Hermitian metric on $L_{D_+}\otimes L_{D_-}$.
     Let $K_\pm=\kappa_\pm(s_\pm,s_\pm)$ and
     $\rho_D:=\kappa(s_D,s_D)\omega$  a $D$-volume form.  
\end{setting}

\begin{theorem}\label{thm:Phiisom}
Using the definitions and notations of Setting \ref{S:input} the following hold:
 
 \begin{enumerate}
     \item 
  Suppose  $D\le0$. Then $s\in A^2(M,\rho_D,E,h)$ implies that $ 
 -D\le \text{\normalfont div}(s).$ Furthermore, the map
 $$
 \Phi_-:A^2(M,\rho_D,E,h)
 \rightarrow A^2(M,\omega,E\otimes L_D, h\otimes\kappa),\quad s\mapsto s\otimes s_D
 $$
 is a surjective linear isometry.

\item Suppose  $0\le D$. The 
 map
$$\Phi_+:A^2(M,\rho_D,E,h)
 \rightarrow A^2_D(M,\omega,E\otimes L_D, h\otimes\kappa),\quad s\mapsto s\otimes s_D
 $$
 is a surjective linear isometry.
 \item Finally, for a general divisor $D \in \text{\normalfont Div}(M)$, the map $$\Phi : A^2(M,\rho_D,E,h) \to A^2_{D_+}(M \setminus \text{\normalfont supp}(D_-), K_-\omega, E\otimes L_{D_+}, h \otimes \kappa_+ ), \quad s \mapsto s\otimes s_{D_+}\big|_{M \setminus \text{\normalfont supp}(D_-)}$$ is a surjective linear isometry.
 
 \end{enumerate}
\end{theorem}

\begin{remark} Note that the space $ A^2_{D_+}(M \setminus \text{\normalfont supp}(D_-), K_-\omega, E\otimes L_{D_+}, h \otimes \kappa_+ )$ agrees with $ A^2_{D_+}(M, K_-\omega, E\otimes L_{D_+}, h \otimes \kappa_+ )$. In what follows however we will use Theorem \ref{theorem:embedding} which formally only applies to the former of the two spaces. \end{remark}

\begin{corollary}\label{cor:simagyokos}
    Let $M, (E,h)$, and $D$  be as in Setting~\ref{S:input}. Suppose in addition that $\rho$ is an arbitrary $D$-volume form. Then  $A^2(M,\rho,E,h)$ is a Hilbert space.
\end{corollary}

\begin{proof}
    By Remark \ref{rem:D-volume}, we may assume $\rho$ to be the $D$-volume form $\rho_D$ defined in Setting \ref{S:input}. Furthermore, let $\omega, D_{\pm}, K_{\pm}$ and $\kappa_{\pm}$ also be as in Setting \ref{S:input}.
    
    Since $K_-\omega$ is a positive continuous volume form on $M\setminus\text{supp}(D_-)$, 
    the Bergman space
 $A^2(M\setminus\text{supp}(D_-),K_-\omega,E\otimes L_{D_+}, h\otimes \kappa_+)$ with the $L^2$ metric is a Hilbert space (cf. \cite{PW94}). $D_+$
 is an effective divisor on $M\setminus\text{supp}(D_-)$ as well.  Then  Theorem~\ref{theorem:embedding}, 
 yields that 
 $A^2_{D_+}(M\setminus\text{supp}(D_-), K_-\omega,E\otimes L_{D_+}, h\otimes \kappa_+)$ is also a Hilbert space and by part 3 of Theorem~\ref{thm:Phiisom} 
 this latter Hilbert space is isometric to $A^2(M,\rho_D,E,h)$ and we are done.
\end{proof}

\begin{proof}[Proof of Theorem \ref{thm:Phiisom}] First we show that both maps $\Phi_{\pm}$ are linear isometries and that they map into the appropriate Bergman space.
 Let $s$ be a holomorphic section of $E$. Then  
\begin{equation}\label{eq:isometry}
\begin{split}
\|s\|^2_{L^2(\rho_D)}=\int_Mh(s,s)\rho_D=
\int_M h(s,s)\kappa(s_D,s_D)\omega=\\  =\int_M h\otimes \kappa(s\otimes s_D,s\otimes s_D)\omega=
\|s\otimes s_D\|^2_{L^2(\omega)}.\hspace{10mm}
\end{split}   
\end{equation} 
Hence
\begin{equation}\label{eq:equiv}
\|s\|^2_{L^2(\rho_D)}<\infty\Longleftrightarrow
\|s\otimes s_D\|^2_{L^2(\omega)}<\infty
\end{equation}
Suppose now that $D\le 0$.
 Since isolated $L^2$ singularities of holomorphic functions are removable, \eqref{eq:equiv} yields  that if $s\in  A^2(M,\rho_D,E,h)$, the section $s\otimes s_D$ extends holomorphically to the points of $\text{supp}(D)$ and so
 \begin{equation}\label{E:D<0}
 0\le \text{div}(\Phi_-(s))= \text{div}(s\otimes s_D)=\text{div}(s)+D\Longrightarrow
 -D\le \text{div}(s).
\end{equation}
 Equation~\eqref{eq:isometry} implies that the map $\Phi_-$   is a  linear isometry into the Bergman space $A^2(M,\omega,E\otimes L_D, h\otimes\kappa)$.

 If $0\le D$, the section $s_D$ is holomorphic
 and $s\otimes s_D$ is as well.
 Furthermore 
 $$\text{div}(\Phi_+(s))=\text{div}(s\otimes s_D)=\text{div} (s)+D\ge D,$$
 since $s$ is holomorphic. This together with
 equation~\eqref{eq:isometry}
yields that $\Phi_+$ indeed maps into  $A^2_D(M,\omega,L\otimes L_D, h\otimes\kappa)$ and that it is a linear isometry.
What is left  to show is that both maps $\Phi_-$ and $\Phi_+$ are surjective. 

Let  $S\in H^0(M,E\otimes L_D)$. Since the bundle $L_D$ is trivial over $M\setminus{\text{supp}(D)}$, there exists a holomorphic section $s$ of $\left.E\right|_{M\setminus{\text{supp}(D)}}$ with
$s\otimes\left.s_D\right|_{M\setminus{\text{supp}(D)}}=\left.S\right|_{M\setminus{\text{supp}(D)}}$.
\begin{prop} \label{prop:smerom}
    $s$ is a meromorphic section of $E\rightarrow M$.
\end{prop}
\begin{proof}
 Let $p\in \text{supp}(D)$. Choose a small open neighborhood $U_p$ of $p$ and let $e_1,\dots,e_r$ and $e$ be holomorphic frames of $E$ and $L_D$ respectively. Then $\left.S\right|_{U_p}=\sum g_je_j\otimes e$ for some $g_j\in\mathcal O(U_p)$ and $\left.s_D\right|_{U_p}=he$ where $h$ is a meromorphic function on $U_p$. On $U^*_p:=U_p\setminus\{p\}$, $s=\sum f_je_j$ for some $f_j\in\mathcal O(U^*_p)$.
 Therefore
 $$
\sum_j g_j\left.(e_j\otimes e)\right|_{U^*_p}= \Big(\sum_j f_je_j\Big)\otimes\left.(he)\right|_{U^*_p}=
\sum_j (hf_j)e_j\otimes \left.e\right|_{U^*_p}
 $$
 and hence $f_j=g_j/h$ is meromorphic
 in $U_p$ for all $j=1,\dots, r$.
\end{proof}
\noindent It follows from Proposition \ref{prop:smerom} that
\begin{equation}\label{eq:div}
 0\le\text{div}(S)=\text{div}(s)+\text{div}(s_D)=\text{div}(s)+D.   
\end{equation}
Now, let $D<0$ and $S\in A^2(M,\omega, E\otimes L_D,h\otimes\kappa)$. Then 
\eqref{eq:div} shows that $s$ is holomorphic and \eqref{eq:isometry}  yields that $s\in A^2(M,\rho_D,E,h)$ 
 proving the surjectivity of $\Phi_-$.

Now, suppose that $D\geq0$ and 
$S\in A^2_D(M,\omega,L\otimes L_D, h\otimes\kappa)$.
Then \eqref{eq:div} shows again that $s$ is a holomorphic section and \eqref{eq:isometry} that
 $s\in A^2(M,\rho_D,E,h)$ 
 proving the surjectivity of $\Phi_+$.

 Finally, let $D \in \text{Div}(M)$ be an arbitrary divisor.  Isolated $L^2$ singularities of holomorphic functions are removable, hence the restriction map
 $$
 \iota: H^0(M,E)\rightarrow H^0(M\setminus\text{supp} (D_-),E),\quad s\mapsto 
 \left.s\right|_{M\setminus\text{supp}(D_-)}
 $$
 induces an onto isometry
 $$
\iota^*:A^2(M,\rho,E,h)=A^2(M,K_+K_-\omega,E,h)\rightarrow A^2(M\setminus\text{supp}(D_-), K_+K_-\omega, E,h)$$
by Lemma \ref{lem:korlap}. Hence, applying part (2) to the effective divisor $D_+$ shows that $\Phi := \Phi_+ \circ \iota^*$ is a surjective linear isometry as well.
\end{proof}

\subsection{The main theorem} We prove the main theorem of this section, which is a direct generalization of \cite{Sz22} and \cite{GGV22} to the case of Bergman spaces associated to $D$-volume forms.

\begin{prop}\label{prop:merom-az-L2}
 Let $M$ be a compact Riemann surface, $(E,h)\rightarrow M$ a Hermitian holomorphic vector bundle with $h$ continuous, $D$ a divisor on $M$, $\rho$ a $D$-volume form and $U\subset M$ open. Define\begin{equation}\label{E:DUdivizor}   
     D_U := \sum_{p \, \in \, \partial U\cap\,\text{\normalfont supp}(D_+)}m_p\cdot p + \sum_{p\, \in \, \overline{U}\cap\, \text{\normalfont supp}(D_-)}m_p\cdot p.
      \end{equation}
We let $H^0(M,E\otimes L_{D_U}):=\{s\in \mathcal M(M,E)\mid \text{\rm div}(s)+D_{U}\ge 0\}$ where $\mathcal{M}(M,E)$ denotes the space of meromorphic sections of $E$. Then
$$H^0(M,E\otimes L_{D_U}) \subseteq A^2(U,\rho,E,h).$$ 
\end{prop}
\begin{proof}
    This follows from the definitions and Corollary~\ref{cor:L2sing}
\end{proof}

\noindent The main theorem (Theorem \ref{thm:gyokos} below) essentially states that the above inclusion is either an equality or has infinite codimension.

\begin{theorem} \label{thm:gyokos}
 Let $M$ be a compact Riemann surface together with a divisor $D$, a $D$-volume form $\rho$, a holomorphic vector bundle $E\ra M$ and a continuous Hermitian metric $h$ on $E$. Then the following are equivalent for any open subset $U \subseteq M$.
    \begin{itemize}
        \item[(a)] $\dim_\bC A^2(U,\rho, E, h) < \infty$.
        \item[(b)] The set $M\setminus U$ is locally polar.
        \item[(c)] $A^2(U,\rho, E, h) = H^0(M,E\otimes L_{D_U})$,\quad where $D_U$ is defined in \eqref{E:DUdivizor}.
    \end{itemize}
\end{theorem}

\noindent Compare to \cite{Sz22,GGV22,GGV24} in the special case where $D=0$. Notice also that (b) is independent of both $D$ and $(E,h)$. It follows that if (a) holds for \emph{some} divisor $D$ and \emph{some} Hermitian bundle $(E,h)$ then both (a) and (c) hold for \emph{any} divisor $D$ and \emph{any} Hermitian bundle $(E,h)$.

\begin{proof}
Throughout the proof we will make use of the definitions and notations of Setting \ref{S:input}.

($a\Longrightarrow b$) In light of Remark~\ref{rem:rhoindep}, $A^2(U,\rho,E,h)=A^2(U,\rho_D,E,h)$, so we can assume that $\rho=\rho_D$.

    Choose small pairwise disjoint  neighborhoods $U_p \subseteq M$ for each $p\in \text{supp}(D_+)$ and modify $K_+$ slightly  to get a function $\widetilde K_+$  with the following properties:
    \begin{enumerate}
        \item $\widetilde K_+>0$
        \item $\widetilde K_+ \geq K_+$
        \item $\widetilde K_+=K_+$ on the set $M \setminus \big( \bigcup_{p\,\in\, \text{supp}(D_+)}U_p \big)$
    \end{enumerate}
    Let $\tilde\omega:=\widetilde K_+\omega$ and $\tilde\rho:=K_-\tilde\omega$. Then $\tilde\omega$  is a  genuine volume form and $\tilde\rho$  a  $D_-$
    volume form on $M$. Consequently $\left.\tilde\rho\right|_U$ will be a $\left.D_-\right|_U$-volume form. Since 
    $D_-\le0$, the first part of Theorem~\ref{thm:Phiisom} applies and yields that
    \begin{equation}\label{eq:csoda}    
    A^2(U,\tilde\rho,E,h)\cong A^2(U,\widetilde K_+\omega,E\otimes L_{D_-},h\otimes\kappa_-).
        \end{equation}

    By our construction $\rho\le\tilde\rho$. Thus 
    $A^2(U,\tilde\rho,E,h) \leq A^2(U,\rho,E,h)$. Since the latter space is finite-dimensional,  
     $A^2(U,\tilde\rho,E,h)$
     must be finite-dimensional as well.
     But then the isomorphism (\ref{eq:csoda})
    yields that
    $A^2(U,\widetilde K_+\omega,E\otimes L_{D_-},h\otimes\kappa_-)$
    is also finite-dimensional. Here however $\widetilde K_+\omega$ is an honest positive continuous volume form on M 
     and (\cite{GGV22}, Theorem 1.1) applies guaranteeing  that $M\setminus U$ is a locally polar set. 
     
     $(b\Longrightarrow c)$
     Since $M \setminus U$
     is locally polar, $M\setminus U= \partial U$, as locally polar sets are nowhere dense. Let $s \in A^2(U,\rho,E,h)$ and $p \in \partial U \setminus \text{supp}(D)$. Since $\rho$ is non-vanishing in a small open neighborhood $V$ of $p$, where $V\setminus U$
   is a polar subset,  Proposition 2.1 in \cite{Sz22} implies that $s$ extends holomorphically to $V$. As this holds for all $p \in \partial U \setminus \text{supp}(D)$, $s$ extends holomorphically over $M \setminus (\partial U\cap \text{supp}(D))$. Denote this extended section also by $s$.
   Now since $\partial U$ is locally polar and consequently it has zero Lebesgue measure, 
   we get that the extended section belongs to the space
   $ A^2(M\setminus(\partial U\cap \text{supp}(D)),\rho,E,h)$. In particular for each 
    point $p \in \partial U \cap \text{supp}(D)$, for a small enough neighborhood $V$ of $p$ (so that $V\cap\text{supp}(D_U)=\{p\}$), the section $s$ belongs to $A^2(V\setminus\{p\},\rho,E,h)$. Now Corollary~\ref{cor:L2sing} implies  $\text{div}(s)+D_U\ge0$.
   We proved  so far that $A^2(U,\rho,E,h) \subseteq \{s\in \mathcal M(M,E)\mid \text{div}(s)+D_{ U}\ge0\}$.  Then 
   Proposition~\ref{prop:merom-az-L2} finishes the proof.

   $(c\Longrightarrow a)$
   The vector space
   $\{s\in \mathcal M(M,E)\mid \text{div}(s)+D_U\ge0\}$ is isomorphic to
   $H^0(M,E\otimes L_{D_U})$ and since $M$ is compact, this latter vector space is finite-dimensional.
\end{proof}

\section{Preliminaries on singular curves}

\subsection{Singularity theory} This subsection is a brief summary of the concepts and notations used throughout the text. For more details on curve singularities see \cite{Wa04}, \cite{GLS07} or \cite{dJP13}.\vspace{3mm}

\noindent In what follows, the word \emph{curve} will mean a connected and reduced complex analytic space of pure dimension one unless otherwise specified. The singular locus of  a curve $X$ will be denoted by $\text{Sing}X$. Given a subset $B\subseteq X$, we will use the notation $B^* = B \setminus \text{Sing}X$.

\begin{definition}
    The \emph{normalization} of a curve $X$ is a (not necessarily connected) smooth curve $\widetilde{X}$ together with a finite and surjective holomorphic map $\pi : \widetilde{X}\to X$ such that its restriction $\pi:\widetilde{X} \setminus \pi^{-1}(\text{Sing}X) \xrightarrow{\, \cong\, } X^*$ is a biholomorphism.

    We remark that it satisfies the universal property that any holomorphic map $M \to X$ from a Riemann surface $M$ factors through $\pi$.
\end{definition}

\begin{prop}\label{prop:puiseux}
    Let $(X,o)$ be an irreducible curve germ. Then there exists a holomorphic germ $\pi: (\bC,0) \to (X,o)$ such that a small enough representative of $\pi$ is the normalization of an open neighborhood of $o \in X$.
\end{prop}

\noindent The map germ of Proposition \ref{prop:puiseux} is called the \emph{Puiseux parametrization} of $X$ around $o$. Gluing together Puiseux parametrizations yields the following

\begin{corollary}
    Any curve $X$ admits a normalization $\widetilde{X} \to X$.
\end{corollary}

 \begin{remark}\label{rem:branches} At every point $p\in X$ the curve $X$ has a local decomposition of the form $(X,p) = \bigcup_{i=1}^r(X_i,p)$ where the germ $(X_i,p)$ is irreducible for all $i$. Since the sets $\overline{\pi^{-1}(X_i\setminus \{p\})}$ and $\pi^{-1}(p)$ have a unique intersection point $\tilde{p}_i \in \widetilde{X}$, the elements of $\pi^{-1}(p)$ correspond bijectively to the set of branches at $p$ via the germ of the normalization $\pi:(\widetilde{X},\tilde{p}_i) \to (X_i,p)$. \end{remark}

\begin{remark} \label{rem:partialnorm}
    A curve $X$ also admits partial normalizations. Let $B \subset X$ be any subset. Then there exists a (not necessarily connected) curve $X^B$ together with a finite and surjective holomorphic map $\pi^B: X^B \to X$ which restricts to the normalization of a neighborhood of $B$ and is a biholomorphism everywhere else.
\end{remark}

\begin{definition}\label{def:intmult}
    Let $(X,o) \subset (\bC^n,o)$ be an irreducible curve germ with Puiseux parametrization $\pi$ and let $f : (\bC^n,o) \to (\bC,0)$ be a map germ defining the hypersurface $(F,o) \subset (\bC^n,o)$. Their \emph{intersection multiplicity} is defined as the number
    $$(X\cdot F)_o := \text{ord}\big(\, f\circ \pi : (\bC,0) \to (\bC,0)\,\big)$$
\end{definition}

\begin{definition}
    Let $(X,o) \subset (\bC^n,o)$ be a curve germ. Its \emph{multiplicity} is defined as
    \begin{align*}
        m(X,o):=& \min \big\{\, (X\cdot F)_o \ \big| \ (F,o) \subset (\bC^n,o) \text{ is a hypersurface} \,\big\}.
    \end{align*}
    It is independent of the embedding $(X,o) \hookrightarrow (\bC^n,o)$ and $m(X,o) = 1$ if and only if $(X,o)$ is nonsingular. If the germ $(X,o)$ has decomposition $(X,o) = (X_1,o)\cup(X_2,o)$ then $m(X,o) = m(X_1,o)+m(X_2,o)$.
\end{definition}

\begin{remark}
    Let $\pi$ denote the Puiseux parametrization of an irreducible curve germ $(X,o) \subset (\bC^n,o)$. In an appropriate local coordinate system it has the form $\pi(t) = \big(t^{m(X,o)}, \pi_2(t),\dots,\pi_n(t)\big)$ with $\text{ord}(\pi_i) > m(X,o)$ for all $i=2,\dots, n$.
\end{remark}

\begin{definition}\label{def:delta}
    Let $(X,o)$ be a curve germ with normalization $\pi$. Its \emph{delta invariant} is defined as the quantity
    $$\delta(X,o) := \dim_\bC \mathcal{O}_{\widetilde{X},\tilde{o}}\big/\pi^*\mathcal{O}_{X,o}.$$
    If $U$ is an open subset of a curve $X$ then $\delta(U)$ is simply defined as the sum of the delta invariants of the singular points of $X$ in $U$.
\end{definition}

    \begin{prop}\label{prop:deltastein}
        Let $U$ be a Stein curve with normalization $\widetilde{U} \to U$. Then
        $$\delta(U) = \dim_\bC  \mathcal{O}(\widetilde{U}) \big/ \pi^*\mathcal{O}(U).$$
    \end{prop}

    \begin{proof}
        Consider the short exact sequence $0 \to \mathcal{O}_U \to \pi_*\mathcal{O}_{\widetilde{U}} \to \pi_*\mathcal{O}_{\widetilde{U}} / \mathcal{O}_U \to 0$ inducing the long exact sequence
    $0 \to \mathcal{O}(U) \to \mathcal{O}(\widetilde{U}) \to H^0(U,\pi_*\mathcal{O}_{\widetilde{U}}/\mathcal{O}_U ) \to H^1(U,\mathcal{O}_U)$
    on sheaf cohomologies. The fourth term vanishes as $U$ is Stein and the statement follows.
    \end{proof}

\begin{remark} \label{rem:vegeskodim}
    The quotient sheaf $\pi_*\mathcal{O}_{\widetilde{U}} / \mathcal{O}_U$ is a coherent skyscraper sheaf supported on $\text{Sing}\,U$. It follows that $\delta(X,p)$ is finite for any point $p \in X$. Even more, taking the tensor product of the above short exact sequence with a holomorphic vector bundle $E$ shows the following: if $\text{Sing}\,U$ is finite, the inclusion $\pi^*\mathcal{O}(U,E) \hookrightarrow \mathcal{O}(\widetilde{U},\pi^*E)$ has finite codimension.
\end{remark}

\subsection{Locally polar sets on singular curves} \label{ss:locpol} In this subsection we compare locally polar sets on a singular curve $X$ to those on its normalization $\widetilde{X}$.

\begin{definition} An upper semicontinuous map $\varphi:X\rightarrow [-\infty, \infty)$ is called weakly subharmonic if for any holomorphic map $f:\Delta\rightarrow X$ from the unit disk $\Delta=\{t\in\bC \mid |t|<1\}$ into $X$, the function $\varphi\circ f$ is either identically $-\infty$, or subharmonic.
\end{definition}
\begin{definition}
 A map $\varphi:X\rightarrow [-\infty, \infty)$ is called subharmonic if it is not identically $-\infty$ on any irreducible component of $X$ and locally it is the restriction of a plurisubharmonic function on an open set $U$ in $\bC^n$ for a local embedding of $X$.
\end{definition}

\begin{remark}\label{rem:FN}
    By (\cite{FN80}, Theorem 5.3.1) these two definitions are in fact equivalent if the function $\varphi$ is not identically $-\infty$.
\end{remark}

    \begin{definition}\label{def:locpol}
   A set $H\subset X$ is called locally polar if for every $p\in H$ there exists an open neighborhood $U\subset X$ of $p$ and a subharmonic function $\varphi:U \to [-\infty,\infty)$ such that $U\cap H\subset \{\varphi=-\infty\}$.
\end{definition}

\begin{theorem}\label{thm:loc-polar-equiv}
Let $X$ be a curve with normalization $\pi:\widetilde X\rightarrow X$. Then a subset $H\subset X$ is locally polar iff  $\widetilde H:=\pi^{-1}(H)\subset\widetilde X$ is locally polar.
\end{theorem}
\begin{proof} ( $\Longrightarrow$ )
Suppose that $H\subset X$ is locally polar.
Let $p\in H$ and $\tilde p\in \pi^{-1}(p)$. Since $\pi$ is a local biholomorphism at smooth points, we can assume that $p$ is a singular point in $X$.  Choose
a small open neighborhood $U$ of $\tilde p$ in $\widetilde X$, so that $U$ is biholomorphic to a disk in $\bC$ and  $U\cap\pi^{-1}(p)=\tilde p$. Since $\left.\pi\right|_U$ is locally biholomorphic at every $z \in U \setminus \{\tilde{p}\}$, the set
$\widetilde H\cap (U\setminus\{\tilde p\})$ is locally polar in $U$. Hence by (\cite{C95}, Ch.21, Proposition 5.5) it is polar. But the union of two polar sets  is again polar (\cite{C95}, Ch.21, Lemma 5.6), so
we get that $\widetilde H\cap U=\{\tilde p\}\cup \widetilde H\cap (U\setminus \{\tilde{p}\})$ is polar as well.

( $\Longleftarrow$ )
Now assume that $\widetilde H\subset\widetilde X$ is locally polar. Let $p\in H$.
Again as before, we can assume that $p$ is a singular point in $X$. We need to find an open neighborhood  $G$ of $p$  with $ G\cap H\subset  G$  polar. Let $G$ be  a small open neighborhood with $(G\setminus \{p\})\cap \text{Sing} X=\emptyset$ and so that $G$ is biholomorphic to an analytic subset of an open set $W\subset\bC^n$ for some  $n$. Let $A_j, j=1,\dots, m$ be the irreducible components of $ G$ and let $\tilde{p}_j$ be the point of $\pi^{-1}(p)$ corresponding to $A_j$ (see Remark \ref{rem:branches}). Composing the normalization $\pi$ with a local coordinate system around $\tilde{p}_j$ in $\widetilde{X}$ we obtain a map $\varphi_j:(V_j,0)\rightarrow (A_j,p)$ where $V_j\subset\bC$ is an open disk centered at $0$. Thus $\varphi_j:V_j\setminus\{0\}\rightarrow A_j\setminus\{p\}$ is biholomorhic, and so $\varphi^{-1}_j\big((A_j\setminus\{p\})\cap H\big)\subset V_j$ is locally polar since $\widetilde{H}$ is locally polar. Hence by (\cite{C95}, Ch.21, Proposition 5.5 ) it is polar and by
(\cite{C95}, Ch.21, Lemma 5.6), 
 $\varphi_j^{-1}(A_j \cap H) = \varphi^{-1}_j((A_j\setminus\{p\})\cap H) \cup \{0\} \subset V_j$
 is also polar in $V_j$. Thus there exists a subharmonic function $\chi_j$ on $V_j$, so that $\varphi^{-1}_j(A_j\cap H)\subset\{\chi_j=-\infty\}$. As $\varphi_j$ is a Puiseux parametrization, $\varphi_j:V_j\rightarrow A_j$ is a homeomorphism.
 Let $v_j:=\chi_j\circ\varphi^{-1}_j$.
 This is then upper semicontinuous. We show that it is weakly subharmonic.
 Let $\psi:\Delta \rightarrow A_j$ be a nonconstant holomorphic function, where $\Delta \subset\bC$ is the unit disk. Then $v_j\circ\psi=\chi_j\circ\varphi^{-1}_j\circ\psi$. The function $\varphi^{-1}_j\circ\psi$ is continuous
 in $\Delta$ and holomorphic in $\Delta \setminus\psi^{-1}(p)$, hence by the Riemann extension theorem, it is holomorphic in $\Delta$. This shows that $v_j$ is weakly subharmonic and by  Remark \ref{rem:FN} it is subharmonic. Also by our construction
 $A_j\cap H\subset\{v_j=-\infty\}$, i.e. $A_j\cap H$ is polar in $A_j$.

 Define $v:G\rightarrow [-\infty,\infty)$ by $v(\zeta):=v_j(\zeta)$, if $\zeta\in A_j$. Now $v$ is well defined, since $v_j(p)=-\infty$ for all $j$. Since
 $\left.v\right|_{A_j}=v_j$ is upper semicontinuous on $A_j$, one can easily see that $v$ is upper semicontinuous on $ G$. Let $\psi:\Delta \rightarrow  G$ be nonconstant, holomorphic, where $\Delta$ is the unit disk. Then $\Delta \setminus\psi^{-1}(p)$ must be connected, hence $\psi(\Delta \setminus\psi^{-1}(p))$ is connected as a subset of $ G\setminus\{p\}=\bigsqcup^m_{j=1}(A_j\setminus\{p\})$. The sets $A_j\setminus\{p\}$ are pairwise disjoint, therefore there exists  $j$ so that $\psi(\Delta \setminus\psi^{-1}(p))\subset A_j\setminus\{p\}$. Hence $\psi(\Delta )\subset A_j$ and so $v\circ\psi=v_j\circ\psi$ that is subharmonic. Refering again  to Remark \ref{rem:FN}, we finally obtain that the weakly subharmonic $v$ is in fact subharmonic. By our construction $ G\cap H\subset\{v=-\infty\}$ i.e. $H$ is locally polar in $X$ and we are done.
\end{proof}

\begin{corollary}\label{rem:locpolgen}
    Let $Y \subset \bC^n$ be an affine algebraic curve with projective closure $X:=\overline{Y}$ and let $\pi : \widetilde{X} \to X$ be the normalization of $X$. Then $H\subset Y$ is locally polar if and only $\pi^{-1}(H) \subset \widetilde{X}$ is locally polar.
\end{corollary} 

\section{Bergman spaces on algebraic curves}

\subsection{The projective case} Throughout this subsection $X$ will denote a compact curve (embedded into some projective space if necessary).

\begin{definition}\label{def:vol}
    A volume form on $X$ is defined as a volume form $\sigma$ on the Riemann surface $X^*$ with the following additional property: for any singular point $p \in \text{Sing}X$ there exists a neighborhood $U\subseteq X$ of $p$ together with an analytic embedding $\iota:U \hookrightarrow \bC^n$ such that $\sigma \sim \iota^*\omega$ on $U^*$ where $\omega$ is the K\"ahler form associated to the Euclidean metric on $\bC^n$. (See \ref{def:gyokosmetrika} for the definition of ``$\sim$''.)
\end{definition}

\begin{remark}\label{rem:volformexistence}
    Such a volume form always exists. One possibility to show this is to glue together volume forms coming from local analytic embeddings using a partition of unity. Another way to see this is to use the fact that any compact complex curve $X$ is projective algebraic, i.e. admits a global analytic embedding $X \hookrightarrow \mathbb{P}^n$. (See e.g. \cite{RuppSera}, Thm 6.2 or \cite{Grauert}, §3, Satz 2.) Now if $\omega_{\text{FS}}$ denotes the Fubini-Study form on $\mathbb{P}^n$, then  $\sigma := \omega_{\text{FS}}\big|_{X^*}$ will be a volume form on $X$ satisfying Definition \ref{def:vol}.
\end{remark}

\noindent Now we are ready to define the central object of this paper, the Bergman space of a singular curve.

\begin{definition}\label{def:singberg}
    Let $X$ be a compact curve with a volume form $\sigma$, a vector bundle $E\ra X$ and a continuous Hermitian metric $h$ on $E$. Given a nonempty  open subset $U \subseteq X$, define its Bergman space to be the vector space
    $$A^2(U,\sigma,E,h) := \Big\{ \, s \in H^0(U,E)  \ \Big| \ \int_{U^*} h(s,s)\sigma < \infty \, \Big\}.$$
\end{definition}

\noindent Let $X$ be a curve with singular locus $\text{Sing}X$ and normalization $\pi : \widetilde{X} \to X$. At every point $p\in X$ the curve $X$ has a local decomposition of the form $(X,p) = \bigcup_i(X_i,p)$ where the germ $(X_i,p)$ is irreducible for all $i$. The multiplicity of the (local) branch $X_i$ at $p$ will be denoted by $\text{m}(X_i,p)$. The preimages $\pi^{-1}(p) = \{\tilde{p}_i\}_i$ correspond bijectively to the set of branches at $p$ via the normalization $\pi:(\widetilde{X},\tilde{p}_i) \to (X_i,p)$. A point $p \in X$ is a smooth point if and only if $X$ has only one (local) branch at $p$ and ${\rm m}(X,p) = 1$. This is also equivalent to $\pi: (\widetilde{X},\tilde{p}) \to (X,p)$ being an isomorphism of germs.

\begin{definition}\label{def:multdiv}
   The multiplicity-divisor of the curve $X$ is defined as
   $$D_\text{m} := \sum_{p\, \in\, \text{Sing}X} \Big( \sum_i \big( \text{m}(X_i,p)-1\big)\cdot \tilde{p}_i \Big) \in \text{Div}(\widetilde{X}).$$
\end{definition}

\begin{prop} \label{prop:projDmvol}
    Let $X$ be a compact and reduced complex analytic curve with volume form $\sigma$ and normalization map $\pi : \widetilde{X} \to X$. Then the form $\rho = \pi^*\sigma$ is a $D_\text{\normalfont m}$-volume form (see Definition \ref{def:gyokosmetrika}) on $\widetilde{X}$.
\end{prop}

\begin{proof}
    The normalization restricts to an isomorphism $\widetilde{X}\setminus \pi^{-1}\text{Sing}X \xrightarrow{\ \cong\ } X^*$ and $\widetilde{X} \setminus \text{supp}(D_\text{m}) \subseteq \widetilde{X}\setminus \pi^{-1}\text{Sing}X$ thus $\rho$ is indeed a volume form on $\widetilde{X}\setminus \text{supp}(D_\text{m})$.

    Locally, around the point $\tilde{p}_{i} \in \widetilde{X}$ the normalization can be identified with the Puiseux parametrization
        $$(\mathbb{C},0)  \xleftrightarrow{\ \, \cong \, \ } (\widetilde{X},\tilde{p}_{i}) \xrightarrow{ \ \, \pi \ \, } (X_i,p) \xhookrightarrow{ \ \ \ \, } (\mathbb{C}^N,o).$$
    In suitable local coordinates it is given by the function $t \mapsto \big(t^{\text{m}(X_i,p)},F(t)\big)$
    where $F : (\bC,0) \to (\bC^{N-1},o)$ is some analytic function with each coordinate function having order strictly greater than $\text{m}(X_i,p)$. A simple computation (see e.g. \cite{RuppSera}, Chap. 3) shows that in this local coordinate system around $\tilde{p}_{i}$ we have
        \begin{equation*} \label{eq:wupp}
            \rho \sim |t|^{2(\text{m}(X_i,p)-1)}\frac{i}{2}dt\wedge d\bar{t}.
        \end{equation*}
        This means precisely that $\rho$ is a $D_\text{m}$-volume form.  
\end{proof} 

In the smooth case, the Bergman space is a Hilbert space. The following theorem states that this will hold in the singular case as well.

\begin{theorem}\label{thm:singhilb}         The vector space $A^2(U,\sigma,E,h )$ is a Hilbert space for any open $U \subseteq X$.
\end{theorem}

\begin{proof}
    Let $\pi: \widetilde{X}\to X$ denote the normalization, $\rho = \pi^*\sigma$, $\widetilde{U} = \pi^{-1}(U)$, $U^* = U\setminus \text{Sing}X$ and $\widetilde{U}^* = \pi^{-1}(U^*)$. The normalization induces an isometric embedding of inner product spaces
    $$\pi^* : A^2(U,\sigma,E,h) \xrightarrow{ \ \cong \ } A^2(\widetilde{U},\rho, \pi^*E,\pi^*h) \cap \pi^*\mathcal{O}_X(U,E) \xhookrightarrow{ \ \ \ } A^2(\widetilde{U},\rho, \pi^*E,\pi^*h).$$
    Since $\rho$ is a $D_\text{m}$-volume form by Proposition \ref{prop:projDmvol}, Corollary \ref{cor:simagyokos} implies that $A^2(\widetilde{U},\rho, \pi^*E,\pi^*h)$ is a Hilbert space whence all we have to show is that $A^2(\widetilde{U},\rho, \pi^*E,\pi^*h) \cap \pi^*\mathcal{O}_X(U,E)$ is a closed subspace.

    Let $(\tilde{s}_n)_{n\in\mathbb{N}} \subset A^2(\widetilde{U},\rho, \pi^*E,\pi^*h) \cap \pi^*\mathcal{O}_X(U,E)$ be a sequence converging to some section $\tilde{s} \in A^2(\widetilde{U},\rho, \pi^*E,\pi^*h)$ in the $L^2$-topology. Each section $\tilde{s}_n$ is the pullback of some section $s_n \in A^2(U,\sigma,E,h)$ and we need to show that so is $\tilde{s}$. Since $\pi : \widetilde{U}^* \to U^*$ is an isomorphism, there exists a section $s \in A^2(U^*,\sigma,E,h)$ with $s\circ \pi = \tilde{s}\big|_{\widetilde{U}^*}$. We will show that $s$ extends holomorphically to all of $U$.
    Let $p\in \text{Sing}X\cap U$, and choose some neighborhood $V\subset U$ around $p$ such that $\text{Sing}X\cap V=\{p\}$ and $E$ is trivial on $V$. Fix a trivialization. This also induces a trivialization of $\pi^*E$ over $\pi^{-1}(V)=\widetilde{V}\subset  \widetilde{X}$. Under this trivialization, we identify the sections $\tilde{s}_n, \tilde{s}$ (resp. $s_n,s$) with vector valued-functions on $\widetilde{V}$ (resp. $V$). These functions are also all $L^2$ with respect to the restricted volume forms and Hermitian forms. In this situation we can use the evaluation maps $\text{Ev}_{p_i,k}^0$ defined in Lemma \ref{lem:kiertfolyt} on the points $p_i\in \pi^{-1}(p)$. The evaluation maps imply two things. First that for all $p_i,p_j\in \pi^{-1}(p)$ we have $\tilde{s}(p_i)=\tilde{s}(p_j)$, i.e. we can extend $s$ as continuous function. The second implication is that the functions $\tilde{s}_n$ converge locally uniformly to $\tilde{s}$ and since $\pi$ is a proper map we also see that the functions $s_n$ converge locally uniformly to $s$.
    On the vector space $\mathcal{O}_X(V)^{\text{rk}(E)}$ the topology of locally uniform convergence coincides with the canonical topology (see \cite{kaupkaup}, Def. 55.1 and E.55j) which is a complete topology. Hence the limit function $s$ is in fact holomorphic on all of $U$. This completes the argument.
\end{proof}

\begin{definition}\label{def:thesheaf}
     Let $X$ be a complex curve with normalization $\pi : \widetilde{X} \to X$ and multiplicity-divisor $D_\text{\normalfont m} \in \text{\normalfont Div}(\widetilde{X})$. Given a subset $B \subseteq X$, define the sheaf $\mathcal{P}_X^B$ by setting
     $$\mathcal{P}_X^B(U) := \{ \, f\in \mathcal{O}_X(U \setminus (\text{Sing}X \cap B)) \ | \ \pi^*f \in \mathcal{M}_{\widetilde X}(\widetilde{U}) \, , \ \text{\normalfont div}(\pi^*f) +D_\text{\normalfont m}\big|_{\widetilde{U}} \geq 0 \, \}$$
     for any open subset $U \subset X$, where $\widetilde{U}:=\pi^{-1}(U)$.
\end{definition}

\noindent The notation $\mathcal{P}_X$ originates in the fact that the sheaf is defined on a $\mathcal{P}$rojective curve $X$. This is to distinguish it from its very similarly defined affine analogue, c.f. Definition \ref{def:theaffsheaf}.

\begin{prop}\label{prop:coherence}
    The sheaf $\mathcal{P}_X^B$ is a coherent analytic sheaf.
\end{prop}

\begin{proof}
    Let $\pi^B : X^B \to X$ denote the partial normalization of $X$ at the points of $\text{Sing}X \cap B$ and let $\nu^B : \widetilde{X} \to X^B$ be the (total) normalization. They fit into a commutative triangle
    \[\begin{tikzcd}
	{X^B} & {\widetilde{X}} \\
	X
	\arrow["{\pi^B}"', from=1-1, to=2-1]
	\arrow["{\nu^B}"', from=1-2, to=1-1]
	\arrow["\pi", from=1-2, to=2-1]
    \end{tikzcd}.\]
    Define the divisor $D_\text{m}^B = \nu^B_*\big(D_\text{m}\big|_{\pi^{-1}(\text{Sing}X \cap B)}\big) \in \text{Div}(X^B)$. This makes sense, as $\nu^B$ is bijective when restricted to $\pi^{-1}(\text{Sing}X\cap B)$. Since $\text{supp}(D_\text{m}^B) \cap \text{Sing}(X^B) = \emptyset$, the sheaf $\mathcal{O}_{X^B}(D_\text{m}^B)$ is a well-defined holomorphic line bundle over $X^B$. It follows that $\pi^B_*\mathcal{O}_{X^B}(D_\text{m}^B)$ is a coherent analytic sheaf over $X$. Finally, one checks that
    $\pi^B_*\mathcal{O}_{X^B}(D_\text{m}^B) = \mathcal{P}_X^B$.
\end{proof}

\noindent The use of the space $X^B$ in the proof of Proposition \ref{prop:coherence} is necessary as it is possible that $\pi_*\mathcal{O}_{\widetilde{X}}(D_\text{m}\big|_{\pi^{-1}(\text{Sing}X \cap B)}) \supsetneq \pi^B_*\mathcal{O}_{X^B}(D_\text{m}^B)$ is a proper subsheaf. It is clear from Definition \ref{def:thesheaf} that 
    $$H^0(X,\mathcal{P}_X^{\partial U}\otimes E) \subseteq A^2(U,\sigma, E, h ).$$
    The main theorem of this section (Theorem \ref{thm:projmain} below) essentially states that the codimension of this inclusion can only attain two values: $0$ and $\infty$.

\begin{theorem}\label{thm:projmain}
    Let $X$ be a compact curve with a volume form $\sigma$, a vector bundle $E\ra X$ and a continuous Hermitian metric $h$ on $E$. Then the following are equivalent for any open subset $U \subseteq X$.
    \begin{itemize}
        \item[(a)] $\dim_\bC A^2(U,\sigma, E, h) < \infty$.
        \item[(b)] The set $X\setminus U$ is locally polar.
        \item[(c)] $A^2(U,\sigma, E, h) = H^0(X,\mathcal{P}_X^{\partial U} \otimes E)$.
    \end{itemize}
\end{theorem}

\noindent Compare to \cite{Sz22} and \cite{GGV22} in the special case where $X$ is smooth. Notice also that (b) is independent of both $\sigma$ and $h$. It follows that if (a) holds for \emph{some} form $\sigma$ and \emph{some} metric $h$ then both (a) and (c) hold for \emph{any} form $\sigma$ and \emph{any} metric $h$.

Theorem \ref{thm:projmain} will be proved through several lemmas below.

\begin{lemma}\label{lem:uaveges}
    $A^2(U,\sigma,E,h)$ is finite-dimensional if and only if $A^2(\widetilde{U},\pi^*\sigma,\pi^*E,\pi^*h)$
    is finite-dimensional.
\end{lemma}

\begin{proof}
    The normalization induces an isometric embedding of inner product spaces
    $$\pi^* : A^2(U,\sigma,E,h) \xrightarrow{ \ \cong \ } A^2(\widetilde{U},\rho, \pi^*E,\pi^*h) \cap \pi^*\mathcal{O}_X(U,E) \xhookrightarrow{ \ \ \ } A^2(\widetilde{U},\rho, \pi^*E,\pi^*h).$$
    The lemma now follows from Remark \ref{rem:vegeskodim}.
\end{proof}

\begin{lemma}\label{lem:singext}
    Let $(X,\sigma)$, and $(E,h) $ be as before. Let $U \subseteq X$ be an open subset such that $A^2(U,\sigma,E,h)$ is finite-dimensional. Then
        $$A^2(U,\sigma,E,h) = A^2(X \setminus (\text{\normalfont Sing}X\cap \partial U),\sigma,E,h).$$
\end{lemma}

\begin{proof}
    Let $s \in A^2(U,\sigma,E,h)$. We will show that $s$ extends holomorphically to all of $X \setminus(\text{\normalfont Sing}X\cap \partial U)$. As before, we will use the notations $\pi : \widetilde{X} \to X$ for the normalization, $\widetilde{U} = \pi^{-1}(U)$ and $\rho = \pi^*\sigma$. By the assumption and the previous Lemma it follows that $A^2(\widetilde{U},\rho,\pi^*E, \pi^*h)$ is finite-dimensional. Theorem \ref{thm:gyokos} then implies that the set $\widetilde{X}\setminus \widetilde{U}$ is locally polar. Hence, by (\cite[Proposition 2.1]{Sz22}, c.f. also \cite{S82}) the section $\tilde{s} = \pi^*s$ extends holomorphically over $\tilde{X} \setminus (\pi^{-1}(\text{Sing}X)\setminus\widetilde{U})  = \widetilde{X} \setminus (\pi^{-1}(\text{Sing}X)\cap \partial \widetilde{U})$. Since $\pi$ is biholomorphic on the set $\widetilde{X}\setminus \pi^{-1}(\text{Sing}X)$, the extension of $\tilde{s}$ yields a holomorphic extension of $s$ to $X \setminus (\text{Sing}X\cap \partial U)$.
\end{proof}

\begin{lemma}\label{lem:main}
    Let $X$ be a compact, connected and reduced complex analytic curve with volume form $\sigma$, normalization $\pi: \widetilde{X} \to X$ and multiplicity-divisor $D_\text{\normalfont m} \in \text{\normalfont Div}(\widetilde{X})$. Let $E\ra X$ be a holomorphic vector bundle with continuous Hermitian form $h$. Given a subset of singular points $S \subseteq \text{\normalfont Sing}X$, let $U = X \setminus S$. Then
    $$A^2(U,\sigma,E,h) = H^0(X,\mathcal{P}_X^{\partial U} \otimes E ).$$
\end{lemma}

\begin{proof}
    As we have seen in the proof Theorem \ref{thm:singhilb}, the normalization induces an isometric embedding of Hilbert spaces
    $$\pi^* : A^2(U,\sigma,E,h) \xrightarrow{ \ \cong \ } A^2(\widetilde{U},\rho,\pi^*E,\pi^*h) \cap \pi^*\mathcal{O}_X(U,E) \xhookrightarrow{ \ \ \ } A^2(\widetilde{U},\rho,\pi^*E,\pi^*h).$$
    Since the union of finitely many point is locally polar, the Hilbert space $A^2(\widetilde{U},\rho,\pi^*E,\pi^*h)$ is finite-dimensional. By Theorem \ref{thm:gyokos} this finite-dimensional Hilbert space is equal to $\{ \, \tilde{s} \in \mathcal{M} (\widetilde{X},\pi^*E) \ | \ {\rm div}(\tilde{s}) + D_{\rm m}\big|_{\pi^{-1}(S)} \geq 0 \, \}$. It follows that $$A^2(U,\sigma,E,h) = H^0(X,\mathcal{P}_X^{\partial U} \otimes E )$$
    as claimed.
    \end{proof}

We have now established the equivalence (a) $\Longleftrightarrow$ (c) of Theorem \ref{thm:projmain} and the next Lemma will establish the equivalence (a) $\Longleftrightarrow$ (b).

\begin{lemma}
    Let $X$ be a connected and reduced complex analytic curve with volume form $\sigma$ and $(E,h)$ a Hermitian vector bundle over $X$. Let $U \subseteq X$ be an open subset. The Hilbert space $A^2(U,\sigma,E,h)$ is finite-dimensional if and only if the the set $X \setminus U$ is locally polar.
\end{lemma}

\begin{proof}
    As usual, the normalization is denoted by $\pi : \widetilde{X} \to X$ and let $\rho = \pi^*\sigma$. As $\pi^*\mathcal{O}_X(U,E) \subseteq \mathcal{O}_{\widetilde{X}}(\widetilde{U},\pi^*E)$ has finite codimension by Remark \ref{rem:vegeskodim}, $A^2(U,\sigma,E,h)$ is finite-dimensional if and only if $A^2(\widetilde{U},\rho,\pi^*E,\pi^*h)$ is finite-dimensional, which is equivalent to $\widetilde{X}\setminus \widetilde{U}$ being locally polar by Theorem \ref{thm:gyokos}. Finally, by Theorem \ref{thm:loc-polar-equiv}, $\widetilde{X}\setminus \widetilde{U}$ is locally polar if and only if $X \setminus U$ is locally polar.
\end{proof}

\subsection{Examples}\label{section:dimension}

Throughout this subsection we illustrate Theorem \ref{thm:projmain} with a couple of examples.

\begin{example}
    Let $X$ be a nodal curve. In this case $D_\text{m} = 0 \in \text{Div}(\widetilde{X})$, thus
    $$A^2(U,\sigma) = \begin{cases}
        \hfil H^0(X,\mathcal{O}_X) \cong \bC & \text{ if } X \setminus U \text{ is polar }\\
        \text{infinite-dimensional} & \text{ otherwise}
        \end{cases}$$
        for any volume form $\sigma$ on $X$ and any open subset $U \subseteq X$.
\end{example}

\begin{example}
    Let $X$ be a unicuspidal rational curve with  $\text{Sing}X = \{p\}$ and multiplicity $m(X,p) = m$. Let $U \subset X$ be an open subset. Then
    $$A^2(U,\sigma) = \begin{cases}
        \hfil H^0(X,\mathcal{O}_X) \cong \bC & \text{ if } X \setminus U \text{ is polar and } p \in U \\
        H^0(X, \pi_*\mathcal{O}_{\mathbb{P}^1}(m-1)) \cong \bC^m & \text{ if } X \setminus U \text{ is polar and } p \not\in U \\
        \hfil \text{infinite-dimensional} & \text{ otherwise}
    \end{cases}$$
    for any volume form $\sigma$ on $X$.
\end{example}

    \begin{definition}
        Let $X$ be a curve with volume form $\sigma$. Given an open subset $U \subseteq X$, its $L^2$ delta invariant is defined as
        $$L^2\delta(U,\sigma) := \dim_\bC A^2(\widetilde{U},\pi^*\sigma ) \big/ \pi^* A^2(U,\sigma).$$
        Compare with Definition \ref{def:delta} and Proposition \ref{prop:deltastein}.
    \end{definition}

    \begin{prop}\label{prop:l2deltaleqdelta}
        Let $X$ be a curve with volume form $\sigma$. Then $$L^2\delta(U,\sigma) \leq \delta(U)$$ for any open subset $U \subseteq X$.
    \end{prop}

    \begin{proof}
        The second isomorphism theorem shows that
        \begin{align*}                  
            A^2(\widetilde{U},\pi^*\sigma)\big/\pi^*A^2(U,\sigma) &= A^2(\widetilde{U},\pi^*\sigma)\big/\big(\,A^2(\widetilde{U},\pi^*\sigma)\cap \pi^*\mathcal{O}_X(U)\,\big) \\
            &\cong \big( \, A^2(\widetilde{U},\pi^*\sigma) + \pi^*\mathcal{O}_X(U) \,) \big/ \pi^*\mathcal{O}_X(U) \\
            &\subseteq \mathcal{O}_{\widetilde{X}}(\widetilde{U}) \big/ \pi^*\mathcal{O}_X(U).
        \end{align*}
        If $U$ is Stein, the dimension of this latter quotient is precisely $\delta(U)$ by Proposition \ref{prop:deltastein}. Remains the case when $X$ is compact and $U=X$ but then both sides are $0$.
    \end{proof}

\noindent One might expect these two invariants to be equal, however this is not the case. The intuitive reason for this is that $\delta(U)$ is determined by $U \cap \text{Sing}X$, while $L^2\delta(U)$ depends on $\overline{U} \cap \text{Sing}X$. 

In the remainder of this subsection we give an example showing that the $L^2$ delta invariant can attain any value between $0$ and the ``usual" delta invariant. This is described more precisely in the following

\begin{prop}
    For any natural number $n$, there exists a Stein curve $U$ with $\delta(U) = n$ such that for any $0 \leq k \leq n$ there is a compact curve $X_k$ with volume form $\sigma_k$ and an open inclusion $U \hookrightarrow X_k$ with the property that
    $$L^2\delta(U,\sigma_k) = k.$$
\end{prop}

\begin{proof} Define $U$ to be the affine algebraic curve given by the equation $x^2-y^{2n+1} = 0$
in the plane $\bC^2$.
It has a single singularity at the origin with delta invariant $n$, whence $\delta(U) = n$. Its normalization is given by $ \nu : \bC \to U\,;\ t \mapsto (t^{2n+1},t^2)$.

Let $V_k$ be the affine algebraic curve given by the equation $u^{2k} - v^{2k+1} = 0$ in the plane $\bC^2$. It has a single singularity at the origin with multiplicity $2k$. Its normalization is given by $ \nu_k : \bC \to V_k\,;\ s \mapsto (s^{2k+1},s^{2k})$.

Finally, let $I : \bC^* \to \bC^*$ denote the reciprocal map $t \mapsto s = 1/t$ and define $X_k$ by gluing together $U$ and $V_k$ via the biholomorphism
$$\nu_k \circ I \circ \nu^{-1} : U\setminus \{(0,0)\} \xrightarrow[]{\ \cong \ } V_k \setminus \{(0,0)\}.$$
The maps $\nu$ and $\nu_k$ ``stick together" to yield a global normalization map $\pi_k : \mathbb{P}^1 \to X_k$ with $\pi_k^{-1}(\text{Sing}X_k) = \big\{ [0:1],[1:0] \big\}$

Now, according to Theorem \ref{thm:projmain}
$$A^2(\widetilde{U}) = H^0\big(\mathbb{P}^1,(2k-1)\cdot [0:1]\big) \cong \bC\langle 1,t,t^2,\dots,t^{2k-1} \rangle.$$
On the other hand
\begin{align*}
    \pi_k^*A^2(U) &= \pi_k^*\mathcal{O}_{X_k}(U)\cap A^2(\widetilde{U})\\
    &\cong \bC[t^2,t^{2n+1}] \cap \bC\langle 1,t,t^2,\dots,t^{2k-1} \rangle \\
    &= \bC\langle\,t^\ell\ | \ \exists\, a,b\in\mathbb{N} : \ell = 2a+(2n+1)b \text{ and } \ell \leq 2k-1 \, \rangle.
\end{align*}
Hence $L^2\delta(U,\sigma) = \#\{\,0\leq \ell \leq 2k-1 \ | \ \ell \neq 2a+(2n+1)b \text{ for any } a,b \in \mathbb{N} \ \} = \min\{k,n\}$
as claimed. \end{proof}

\subsection{The affine case}\label{ss:aff}

 In this subsection we adapt Theorem \ref{thm:projmain} to the setting of affine algebraic curves. Any holomorphic vector bundle over a Stein curve is analytically trivial (Corollary 8.3.3. of \cite{F17}, see also \cite{G571} and \cite{G572}). Hence, in the following we may restrict ourselves to the case of the trivial line bundle without loss of generality.

Let $Y \subset \bC^n$ be an affine \textit{algebraic} curve with projective closure
$$X:=\overline{Y} = Y \sqcup (H_\infty\cap X) \subset \mathbb{P}^n$$
where $H_\infty$ is the hyperplane at infinity. Hence $\text{Sing}X \subseteq \text{Sing}Y \cup (H_\infty \cap X)$. Let $\pi : \widetilde{X} \to X$ denote the normalization of $X$. At every point $p\in X$ the curve $X$ has a local decomposition of the form $(X,p) = \bigcup_i(X_i,p)$ where the germ $(X_i,p)$ is irreducible for all $i$. The preimages $\pi^{-1}(p) = \{\tilde{p}_i\}_i$ correspond bijectively to the set of branches at $p$ via the normalization $\pi:(\widetilde{X},\tilde{p}_i) \to (X_i,p)$.

\begin{definition}\label{def:affmultdiv}
   The affine multiplicity-divisor of the curve $Y$ is defined as
\begin{align*}
    D^{\mathbb{A}}_\text{m} :=& \sum_{p\,\in \,\text{Sing}Y} \sum_i (m(X_i,p)-1)\cdot \tilde{p}_{i}  -\sum_{p \, \in \, H_\infty \cap X} \sum_i \big((X_i\cdot H_\infty)_{p}+1\big)\cdot \tilde{p}_{i}\in \text{Div}(\widetilde{X})
\end{align*}
   where $(X_i\cdot H_\infty)_{p}$ denotes the intersection multiplicity of the curve germ $X_i$ with the hyperplane $H_\infty$ at the point $p$. Note that the degree of $D_\mathbf{m}^\mathbb{A}$ is
   $$\deg(D_\mathbf{m}^\mathbb{A}) = \sum_{p\in Y}\big(m(X,p)-r(X,p)\big)-\sum_{p \in X \cap H}\big((X\cdot H)_p + r(X,p) \big)$$
   where $r(X,p)$ denotes the number of locally irreducible components of $X$ at the point $p$.
\end{definition}

Let $Y \subset \bC^n$ be an affine algebraic curve with volume form $\sigma = \left.\omega\right|_Y$ where $\omega$ is the K\"ahler form $\omega = \frac{i}{2}\sum_idz_i \wedge d\bar{z}_i$ on $\bC^n$.

\begin{prop}
    Let $Y$, $\sigma$ and $\pi : \widetilde{X}\to X$ be as defined above. Then $\rho := \pi^*\sigma$ is a $D_\text{m}^\mathbb{A}$-volume form on $\widetilde{X}$.
\end{prop}

\begin{proof}
    Due to Proposition \ref{prop:projDmvol}, it is sufficient to analyze $\rho$ in the neighborhood of points $\tilde{p} \in \pi^{-1}(H_\infty \cap X)$. In the chart $(z_1,\dots,z_N) \leftrightarrow [z_1:\dots:z_N:1]$ we have $\omega = \frac{i}{2}\sum_idz_i\wedge d\bar{z}_i$. Assume that the point $p$ lies in the chart $(w_1,\dots,w_N) \leftrightarrow [1:w_1:\dots:w_N]$. In this coordinate system $\omega$ has the form
        $$\omega = \frac{i}{2}\Bigg( d\bigg( \frac{1}{w_N}\bigg)\wedge d\bigg( \frac{1}{\bar{w}_N}\bigg) + \sum_{j=1}^{N-1}d\bigg( \frac{w_j}{w_N}\bigg)\wedge d\bigg( \frac{\bar{w}_j}{\bar{w}_N}\bigg) \Bigg).$$
    Using a linear transformation if necessary, we may assume that $p=[1:0:\dots:0]$. Let $(X,p)= \bigcup_i(X_i,p)$ be the local irreducible decomposition at $p$ and let $\pi : (\widetilde{X},\tilde{p}_i) \to (X_i,p)$ be the restriction of the normalization to one of the branches. In the above coordinate system this is of the form
    $$\pi(t) = \big[ 1 : f_1(t): \dots : f_N(t) \big]$$
    for some holomorphic functions $f_j : (\bC,0) \to (\bC,0)$. Let $m_j\geq 1$ be the vanishing order of $f_j$ such that $f_j(t) = t^{m_j}g_j(t)$ for some holomorphic function $g_j$ with $g_j(0)\neq 0$. It follows that $f_j'(t) = t^{m_j-1}h_j(t)$ for some holomorphic function $h_j$ with $h_j(0) \neq 0$. One now computes that locally around $\tilde{p}_i \in \widetilde{X}$, the form $\rho = \pi^*\sigma$ can be expressed as
    \begin{align*}
        \rho(t) &= \frac{i}{2}\Bigg( d\bigg( \frac{1}{f_N}\bigg)\wedge d\bigg( \frac{1}{\bar{f}_N}\bigg) + \sum_{j=1}^{N-1}d\bigg( \frac{f_j}{f_N}\bigg)\wedge d\bigg( \frac{\bar{f}_j}{\bar{f}_N}\bigg) \Bigg)\\
        &= \Bigg( \frac{|f'_N|^2}{|f_N|^4} + \sum_{j=1}^{N-1} \bigg| \frac{f'_jf_N-f_jf'_N}{f_N^2} \bigg|^2  \Bigg) \frac{i}{2} dt\wedge d\bar{t}\\
        &= |t|^{-2(m_N+1)}\Bigg( \frac{|h_N|^2}{|g_N|^4}+\sum_{j=1}^{N-1}|t|^{2m_j}\frac{|h_jg_N-g_jh_N|^2}{|g_N|^4}  \Bigg)\frac{i}{2}dt\wedge d\bar{t}\\
        &\sim c \cdot |t|^{-2(m_N+1)}\frac{i}{2}dt\wedge d\bar{t}
    \end{align*}
    where $c =\frac{|h_N(0)|^2}{|g_N(0)|^4} > 0$. Finally, we have that $(X_i,H_\infty)_p= m_N$ as claimed: since the hyperplane $H_\infty$ is given by the equation $w_N=0$, it follows that $(X_i,H_\infty)_p = \text{ord}_{t=0}f_N(t) = m_N$ by Definition \ref{def:intmult}. This is exactly what we wanted to show.
\end{proof}

    \noindent As a corollary to this observation we obtain the main theorem of this section which is the affine algebraic analogue of Theorem \ref{thm:projmain}. Before that however, we introduce the affine version of the sheaf $\mathcal{P}_X^{\partial U}$ introduced in \ref{def:thesheaf}.

    \begin{definition}\label{def:theaffsheaf}
    Let $Y$ be an affine algebraic curve with affine multiplicity divisor $D_\text{m}^\mathbb{A} \in \text{Div}(\widetilde{X})$. Given a subset $B \subseteq Y$, define the sheaf $\mathcal{A}_X^B$ by setting
     $$\mathcal{A}_X^B(U) := \{ \, f\in \mathcal{O}_X(\,U \setminus ((\text{Sing}Y \cap B)\cup H_\infty))\,) \ | \ \pi^*f \in \mathcal{M}_{\widetilde X}(\widetilde{U}) \, , \ \text{\normalfont div}(\pi^*f) +D_\text{\normalfont m}^\mathbb{A}\big|_{\widetilde{U}} \geq 0 \, \}$$
     for any open subset $U \subset X = \overline{Y}$, where $\widetilde{U}:=\pi^{-1}(U)$.
\end{definition}
\begin{remark}
    We use the notation $\mathcal{A}_X$ to emphasize that we are dealing with the $\mathcal{A}$ffine case.
\end{remark}
\begin{prop}
    The sheaf $\mathcal{A}_X^B$ is a coherent analytic sheaf.
\end{prop}

\begin{proof}
    Identical to the proof of Proposition \ref{prop:coherence} except for the fact that one needs to use the partial resolution of $X = \overline{Y}$ at the singular points of the set $B \cup (H_\infty\cap X)$.
\end{proof}

   \begin{theorem}\label{thm:affmain}
    Let $Y \subset \bC^n$ be a reduced algebraic curve with the volume form $\sigma = \omega\big|_Y$ where $\omega$ is the Euclidean K\"ahler form on $\bC^n$. Let $U \subseteq Y$ be an open subset. The following are equivalent.
    \begin{itemize}
        \item[(a)] $\dim_\bC A^2(U,\sigma) < \infty$.
        \item[(b)] The set $Y\setminus U$ is locally polar.
        \item[(c)] $A^2(U,\sigma) = H^0(X,\mathcal{A}_X^{\partial U})$.
    \end{itemize}
\end{theorem}

\begin{remark}
    In fact, under the hypotheses of Theorem \ref{thm:affmain}, for any $f \in A^2(U,\sigma)$, the meromorphic function $\pi^*f$ extends holomorphically to all of $\widetilde{X} \setminus \pi^{-1}(\text{Sing}(Y) \cap \partial U)$ with $\pi^*f \big|_{\pi^{-1}(X\cap H_\infty)} = 0$. Note however that this only implies that $f$ will extend to $X \setminus ({\rm Sing}Y\cap \partial U)$ as a \emph{weakly} holomorphic function and hence we can only conclude that $f \in \mathcal{O}_X(Y \setminus ({\rm Sing}Y\cap \partial U))$.
\end{remark}

\begin{proof}
    We first show the equivalence of (a) and (b). We will temporarily use the notations $A^2_Y(U,\sigma)$ and $A^2_X(U,\sigma)$ depending on whether $U$ is considered as an open subset of $Y$ or of $X=\overline{Y}$. This is just a conceptual difference however as clearly $A^2_Y(U,\sigma) = A^2_X(U,\sigma)$. The normalization induces an isomorphism
    $$\pi^* : A^2_X(U,\sigma) \xrightarrow{\ \cong \ } A^2(\widetilde{U},\rho) \cap \pi^*\mathcal{O}_X(U) \xhookrightarrow{ \ \ \ } A^2(\widetilde{U},\rho).$$
    As $\pi^*\mathcal{O}_X(U) \subseteq \mathcal{O}_{\widetilde{X}}(\widetilde{U})$ has finite codimension, this means that
    $$\dim_\bC A^2_Y(U,\sigma) < \infty \ \Longleftrightarrow \ \dim_\bC A^2(\widetilde{U},\rho) < \infty.$$
    Hence, by Theorem \ref{thm:gyokos}, we have that $\dim_\bC A^2_Y(U,\sigma) < \infty$ if and only if $\widetilde{X} \setminus \widetilde{U}$ is locally polar. Now, apply Corollary \ref{rem:locpolgen} to see that $\widetilde{X} \setminus \widetilde{U}$ is locally polar if and only if $Y \setminus U$ is locally polar.

    Finally, $(c) \Rightarrow (a)$ is clear, while $(a) \Rightarrow (c)$ follows from Lemma \ref{lem:singext} using an argument identical to that of Lemma \ref{lem:main} but applied to the affine multiplicity divisor $D_{\text{m}}^{\mathbb{A}}$.
\end{proof}

\begin{corollary}\label{cor:szingWD}

    Let $Y\subset \bC^n$ be an affine algebraic curve with projective closure $X\subset \mathbb{P}^n$ and let $H_\infty \leq \mathbb{P}^n $ denote the hyperplane at infinity. Assume furthermore, that
    $$\sum_{p\,\in\, Y}\big(m(X,p)-r(X,p)\big)<\sum_{p \,\in\, X \cap H_\infty}\big((X\cdot H)_p + r(X,p) \big).$$
    Then $Y$ has the property that $A^2(U,\sigma)$ is either infinite-dimensional or trivial for any open subset $U \subseteq Y$.
\end{corollary}

\begin{proof}
    Let $\pi: \widetilde{X} \to X$ denote the normalization. If $U \subseteq Y$ has locally polar complement, then $$A^2(U,\sigma) = H^0(X,\mathcal{A}_X^{\partial U}) \subseteq H^0\big(\widetilde{X},\mathcal{O}_{\widetilde{X}}(D_\mathbf{m}^{\mathbb{A}}\big|_{\pi^{-1}\big((\text{Sing}X \,\cap\, \partial U )\,\cup\, H_\infty \big)})\big) = \{0\}$$
    where the first equality is Theorem \ref{thm:affmain}, the inclusion follows from the inclusion of sheaves $\mathcal{A}_X^{\partial U} \subseteq \pi_*\mathcal{O}_{\widetilde{X}}(D_\mathbf{m}^{\mathbb{A}}\big|_{\pi^{-1}\big((\text{Sing}X \,\cap\, \partial U )\,\cup\, H_\infty \big)})$ and the last equality follows from the inequality $$\deg(D_\mathbf{m}^{\mathbb{A}}\big|_{\pi^{-1}\big((\text{Sing}X \,\cap\, \partial U )\,\cup\, H_\infty \big)}) \leq \deg(D_\mathbf{m}^\mathbb{A})$$
    and the assumption of the theorem which means precisely that $\deg(D_\mathbf{m}^\mathbb{A}) < 0$.
\end{proof}

\begin{corollary}\label{cor:simaffin}
    Let $M \subset \bC^n$ be a smooth algebraic curve and $U \subseteq M$ an open subset. Then $A^2(U,\sigma)$ is either infinite-dimensional or trivial.
\end{corollary}

\noindent Corollary \ref{cor:simaffin} is a direct generalization of \cite{Wiegerinck} where this statement was proven in the special case $M=\bC$. Since our proof relies heavily on the fact that $M \subset \bC^n$ is \emph{algebraic}, we propose the following

\begin{conj}\label{conj:analytic}
    Let $M \subset \bC^n$ be a smooth analytic curve and $U \subseteq M$ an open subset. Then $A^2(U,\sigma)$ is either infinite-dimensional or $A^2(U,\sigma) = \{0\}$.
\end{conj}

\noindent The condition on the singularities of $Y$ imposed in Theorem \ref{cor:szingWD} cannot be dispensed with as the following example shows. It follows that the Wiegerinck conjecture cannot be generalized to all Stein varieties.

\begin{example}
    Let $Y$ be an affine algebraic curve with projective closure $X$. If $U:= Y \setminus \text{Sing}Y$, then we have an equality of sheaves $\mathcal{A}_X^{\partial U} = \mathcal{A}_X^{\text{Sing}Y}=\pi_*\mathcal{O}_{\widetilde{X}}(D_\mathbf{m}^\mathbb{A})$, whence Theorem \ref{thm:affmain} implies that $A^2(U,\sigma) = H^0(\widetilde{X},\mathcal{O}_{\widetilde{X}}(D_\mathbf{m}^\mathbb{A}))$.

    In \cite{CPS14}, Calabri, Paccagnan and Stagnaro construct a projective curve $X \subset \mathbb{P}^2$ of degree $36$, having $375$ singular points of type $\big(\{x^2-y^3=0\},o\big) \subset (\bC^2,o)$, and no other singularities. Set $Y := X \setminus H$ for some generic line $H \leq \mathbb{P}^2$.

    As $H$ was chosen generically, it intersects $X$ transversally in $36$ distinct smooth points. Hence $$\deg(D_\mathbf{m}^\mathbb{A}) = \sum_{p \, \in \, \text{Sing}Y}( m(Y,p)-1)- \sum_{p\,\in\,X\cap H}\big( (X\cdot H)_p+1\big) = 375\cdot(2-1) -36\cdot(1+1)=303.$$
    The genus of the normalization is $g(\widetilde{X})=\frac{1}{2}\cdot 35 \cdot 34 -375 = 220$ by the Plücker formula. Finally, applying Riemann-Roch, we obtain
    $$h^0(\widetilde{X},\mathcal{O}_{\widetilde{X}}(D_\mathbf{m}^\mathbb{A})) \geq 1-g(\widetilde{X}) + \deg(D_\mathbf{m}^\mathbb{A}) = 1-220+303 = 84.$$
    It follows that $A^2(U,\sigma)$ is neither infinite-dimensional nor trivial.
    
\end{example}

\end{document}